\documentclass[11pt,a4paper]{article}
\usepackage{latexsym,amssymb,amsmath,amsthm}
\usepackage{locenglish,math}
\usepackage{stmaryrd}
\usepackage{mdwlist}
\usepackage{scrtime}
\usepackage{fourier}
\selectfont
\usepackage[usenames]{color}
\usepackage{color}
\definecolor{black}{rgb}{0.0,0.0,0.0}
\definecolor{white}{rgb}{1.0,1.0,1.0}
\definecolor{vertfonce}{rgb}{0.20, 0.46, 0.25}
\definecolor{rougefonce}{rgb}{0.64, 0.09, 0.20}
\definecolor{bleu}{rgb}{0.0,0.1,1.0}
\usepackage[colorlinks,citecolor=blue,urlcolor=blue,pagebackref=true]{hyperref}
\newcommand{\souligner}[1]{\textbf{#1}}

\newcommand{\LastUpdate}{\eurtoday\ at \thistime}
\usepackage{tikz}
\usetikzlibrary{arrows,positioning}
\usepackage{setspace}
\makecompactlist{itemizec}{itemize}
\makecompactlist{enumeratec}{enumerate}
\begin{document}
\begin{center}
\textbf{\large{\textsf{Type transition of  simple random walks\\
 on  randomly directed  regular lattices}}\footnote{Supported in part by the Italian national PRIN project \textit{Random fields, percolation, and stochastic evolution of systems with many components} and by the  \textit{Actions Internationales} programme of the \textit{Universit\'e de Rennes 1}. M.C.\ acknowledges support from G.N.A.M.P.A. This work has been completed while D.P.\ was on  sabbatical leave from his host university at the \textit{Institut Henri Poincar\'e}. 
}}
\vskip1cm
\parbox[t]{14cm}{
Massimo {\sc  Campanino}$^{a}$ and Dimitri {\sc Petritis}$^b$\\
\vskip5mm   
{\scriptsize  
\baselineskip=5mm
a. Dipartimento di Matematica, Università degli Studi di Bologna,\\
piazza di Porta San Donato 5, I-40126 Bologna, Italy, massimo.campanino@unibo.it
\vskip5mm
b. Institut de Recherche
Math\'ematique, Universit\'e de Rennes I and CNRS UMR 6625\\
Campus de Beaulieu, F-35042 Rennes Cedex, France, dimitri.petritis@univ-rennes1.fr}\\
\vskip1cm
{\small
\centerline{\LastUpdate}
\vskip1cm
\baselineskip=5mm
{\bf Abstract:}
Simple random walks on a partially directed version of $\BbZ^2$ are considered. More precisely, vertical edges between neighbouring vertices of $\BbZ^2$ can be traversed in both directions (they are undirected) while horizontal edges are one-way. The horizontal orientation is prescribed by a random perturbation of a periodic function, the perturbation probability  decays according to a power law in the absolute value of the ordinate.    
We study the type of the simple random walk, i.e.\ its being recurrent or transient, and show that there exists a critical value of the decay power, above which it is almost surely recurrent and below which it is almost surely transient.
}}
\end{center}


\section{Introduction}
\subsection{Motivations}

We study simple symmetric random walks (i.e.\ jumping with uniform probability to one of the available neighbours of a given vertex) on partially directed regular sublattices of $\BbZ^2$ obtained from $\BbZ^2$ by imposing horizontal lines to be uni-directional.   
Although random walks on partially directed lattices have been introduced long-time ago to study the hydrodynamic dispersion of a tracer particle in a porous medium \cite{MatheronMarsily1980} very little was known on them beyond some computer simulation heuristics \cite{Redner1997}. Therefore,  it arose as a surprise for us that so little was rigourously known when we first considered simple random walks on partially directed 2-dimensional lattices in \cite{CamPet-rwrol,CamPet-wien}. In those papers, we determined the type of simple random walks on  lattices obtained from $\BbZ^2$ by keeping vertical edges bi-directional while horizontal edges become one-way. Depending on how  the horizontal allowed direction to the left or the right is determined we obtain dramatically different behaviour \cite[Theorems 1.6, 1.7, and 1.8]{CamPet-rwrol} (these results are  reproduced --- for completeness --- as theorem \ref{thm:old-rwrol} in the present paper).

This result triggered several developments by various authors. In \cite{Guillotin-PlantardLe2007},  the orientation is chosen by means of a correlated sequence or by a dynamical system; in both cases, provided that some  variance condition holds, almost sure transience is established and, in \cite{Guillotin-PlantardLe2008}, a functional limit theorem is obtained. In \cite{Pene2009}, the case of orientations chosen according to a stationary sequence is treated. In \cite{Pete2008}, our results of \cite{CamPet-rwrol,CamPet-wien} are used to study corner percolation on $\BbZ^2$. In \cite{Loynes2012a}, the Martin boundary of these walks has been studied for the models that are transient and proved to be trivial, i.e.\ the only positive harmonic functions for the Markov kernel of these walks are the constants. In \cite{DevulderP`ene2013} a model where the horizontal directions are chosen according to an arbitrary (deterministic or random) sequence but the probability of performing a horizontal or vertical move is not determined by the degree but by a sequence of non-degenerate random variables is considered and shown to be a.s.\ transient.

It is worth noting that all the previous directed lattices are regular in the sense that both the inward and the outward degrees are constant (and equal to 3) all over the lattice. Therefore, the dramatic change of type is due only to the nature of the directedness. However, the type result was always either recurrent or transient. The present paper provides an example where the type of the random walk is determined through the tuning of  a parameter controlling the overall number of defects; it improves thus the insight we have on  those non reversible random walks. Let us mention also that beyond their theoretical interest (a short list
of problems remaining open in the context of such random walks is given in the conclusion section), directed random walks are much more natural models of propagation on large networks like internet than reversible ones.


\subsection{Notation and definitions}

 \textit{Directed} versions of $\BbZ^2$ are obtained as follows: let $\bu=(u_1,u_2)$ and $\bv=(v_1, v_2)$ be arbitrary elements of $\BbZ^2$ and suppose  that a sequence of $\{-1,1\}$-valued variables $\bom{\varepsilon}=(\varepsilon_y)_{y\in\BbZ}$ is given. The pair $(\bu,\bv)\in\BbZ^2\times \BbZ^2$ is an allowed edge to the lattice if either $[u_2=u_1]\wedge [v_2=v_1\pm 1]$ or $[v_2=v_1] \wedge [u_2=u_1+\varepsilon_{v_1}]$. The directed sublattice of $\BbZ^2$ depends obviously of the choice of the sequence $\bom{\varepsilon}$; we denote this partially directed lattice by $\BbZ_{\bom{\varepsilon}}^2$. The choice of  $\bom{\varepsilon}$ can be deterministic or random and will be specified later.

\begin{defi}
A \souligner{simple random walk} on $\BbZ_{\bom{\varepsilon}}^2$
is a $\BbZ^2$-valued Markov chain 
$(\bM_n)_{n\in\BbN}$ with transition probability matrix $P$
having as  matrix elements
\[P(\bu,\bv)=\BbP(\bM_{n+1}=\bv|\bM_n=\bu)=
\left\{\begin{array}{ll}
\frac{1}{3} & \textrm{if }\ \ (\bu,\bv) \textrm{ is an allowed edge of } \BbZ_{\bom{\varepsilon}}^2,\\
0 & \textrm{otherwise.}
\end{array}\right.\]
\end{defi}

\begin{remn} The Markov chain 
$(\bM_n)_{n\in\BbN}$ \textit{cannot be reversible}. Therefore, all the powerful techniques
based on the analogy with electrical circuits (see \cite{DoyleSnell1984,Soardi1994} for modern exposition) or 
spectral properties of graph Laplacians\footnote{Connections of graph Laplacians with electric circuits is at least as old as reference \cite{Weyl1923}, connections with modern cohomology can be found in \cite{Soardi1994,Woess2000,JorgensenPearse2008} but again ideas are much older \cite{Flanders1971b,Flanders1972} and references therein.} \cite{Biggs1974,CvetkovicDoobSachs1995,Chung1997} do not apply.  
\end{remn}

Several $\bom{\varepsilon}$-horizontally directed lattices have been introduced in \cite{CamPet-rwrol},  where the following theorem has been established.
\begin{theo} {\cite[see theorems 1.6, 1.7, and 1.8]{CamPet-rwrol}}
\label{thm:old-rwrol}
Consider a $\BbZ^2_{\bom{\varepsilon}}$ directed lattice. 
\begin{enumeratec}
\item If the lattice is alternatively directed, i.e.\ $\varepsilon_y =(-1)^y$, for $y\in \BbZ$, then the simple random walk on it is recurrent.
\item If the lattice has directed half-planes i.e.\ $\varepsilon_y=-\id_{]-\infty, 0[}(y) + \id_{[0,\infty[}(y)$, 
 then the simple random walk on it is transient.
 \item If  $\bom{\varepsilon}$ is a sequence of $\{-1,1\}$-valued random variables, independent and identically distributed with uniform probability, the simple random walk on it is transient for
 almost all possible choices of the horizontal directions.
 \end{enumeratec}
 \end{theo}
 
 Notice that the above simple random walks are defined on topologically non-trivial directed graphs in the sense that $\lim_{N\rightarrow\infty}\frac{1}{N}\sum_{y=-N}^N \varepsilon_y=0.$
For the two first cases, this is shown by a simple calculation
and for the third case this is an almost sure statement
stemming from the independence of the sequence $\bom{\varepsilon}$.
The above condition guarantees that  transience is not a trivial consequence of a non-zero drift but an intrinsic property of the walk in spite of  its jumps being
statistically symmetric.


\subsection{Results}

In this paper, we consider again a $\BbZ^2_{\bom{\varepsilon}}$ lattice but the sequence $\bom{\varepsilon}$ is specified as follows.

\begin{defi}{}
Let $f:\BbZ\rightarrow \{-1,1\}$ be a $Q$-periodic function with some even integer $Q\geq 2$ verifying $\sum_{y=1}^Q f(y)=0$ and $\bom{\rho}=(\rho_y)_{y\in\BbZ}$ a Rademacher
sequence, i.e.\ 	a sequence  of independent and identically distributed $\{-1,1\}$-valued random variables having uniform distribution on $\{-1,1\}$.
Let $\bom{\lambda}=(\lambda_y)_{y\in\BbZ}$
be a $\{0,1\}$-valued sequence of independent random variables and independent of $\bom{\rho}$ and suppose there exist constants $\beta$ (and $c$) such that 
$\BbP(\lambda_y=1)=\frac{c}{|y|^\beta}$ for large $|y|$. We define the horizontal orientations $\bom{\varepsilon}=(\varepsilon_y)_{y\in\BbZ}$ through $\varepsilon_y= (1-\lambda_y)f(y)+\lambda_y \rho_y$. Then the $\BbZ^2_{\bom{\varepsilon}}$-directed lattice defined above is termed a \textbf{randomly horizontally directed lattice with randomness decaying in power $\beta$}.
\end{defi} 
\begin{theo}
\label{thm:main}
Consider the horizontally directed lattice $\BbZ^2_{\bom{\varepsilon}}$ with randomness decaying in power $\beta$. 
\begin{enumeratec}
\item If $\beta <1$ then the simple random walk is transient for almost all realisations of the sequence $(\lambda_y, \rho_y)$.
\item If $\beta >1$ then the simple random walk is recurrent for almost all realisations of the sequence $(\lambda_y, \rho_y)$.
\end{enumeratec}
\end{theo}

\begin{remn}
It is worth noting that the periodicity of the function $f$ is required only to prove recurrence; for proving transience, any function $f$ can be used. 
\end{remn}
\begin{remn}
 In the previous model, the levels $y$ where $\lambda_y\ne 0$, can be viewed as random defects perturbing a periodically directed model whose horizontal directions are determined by the periodic function $f$. Thus, it is natural to consider the random variable $\|\bom{\lambda}\|:=\|\bom{\lambda}\|_1=\card\{y\in\BbZ: \lambda_y=1\}$ as the \textbf{strength} of the perturbation.  
When $\beta>1$, by Borel-Cantelli lemma, $\|\bom{\lambda}\|<\infty$ a.s., meaning that there are a.s.\ finitely many levels $y$ where the horizontal direction is randomly perturbed with respect to the direction determined by the periodic function;
if $\beta<1$, then $\|\bom{\lambda}\|=\infty$ a.s.  
\end{remn}
An extreme choice of ``random'' perturbation is when $\bom{\lambda}$ is a deterministic $\{0,1\}$-valued sequence. We have then the following 

\begin{prop}
\label{pro:deterministic}
When  $\bom{\lambda}$ is a deterministic $\{0,1\}$-valued sequence with $\|\bom{\lambda}\|<\infty$, then  the simple random walk is recurrent.
\end{prop}
Note however that the previous proposition does not provide us with a  necessary condition  for recurrence. We shall give  in \S\ref{sec:open} the following 
\begin{coex}{}There are deterministic $\{0,1\}$-sequences $\bom{\lambda}$, with $\|\bom{\lambda}\|=\infty$ (infinitely many deterministic defects), leading nevertheless to recurrent random walks.
\end{coex}


\section{Technical preliminaries}
\label{sec:technical}
Since the general framework  developed in \cite{CamPet-rwrol} is still useful, we only recall here the basic facts. 
It is always possible to choose a sufficiently large abstract probability space
$(\Omega,\cA,\BbP)$ on which are defined  
all the sequences of random variables we shall use, namely $(\rho_y), (\lambda_y)$, etc.\  and in particular the Markov chain $(\bM_n)_{n\in\BbN}$ itself. 
When the  initial probability of the chain is $\nu$, then obviously $\BbP:=\BbP_\nu$. When $\nu=\delta_x$ we write simply $\BbP_x$ instead of $\BbP_{\delta_x}$.

The idea of the proof is to decompose  
the stochastic process $(\bM_n)_{n\in\BbN}$ into a vertical skeleton --- obtained by the vertical projection of $(\bM_n)$ that is stripped out of the waiting times corresponding to the horizontal moves ---
and a horizontal
component. More precisely,
define $T_0:=0$ and for $k\geq 1$ recursively: $T_k=\inf\{n>T_{k-1}: \scalar{\bM_n-\bM_{n-1}}{\be_2}\ne 0\}$. Introduce then the sequences $\psi_k=\scalar{\bM_{T_k}-\bM_{T_{k-1}}}{\be_2}$ for $k\geq 1$, and $Y_n=\sum_{k=1}^n \psi_k$, for $n\geq 0$ (with the convention $Y_0=0$). The process $(Y_n)$ is a simple random walk on the vertical axis, called the \souligner{vertical skeleton}; its occupation measure of level $\{y\}$ is denoted $\eta_n(y)=\sum_{k=0}^n \id_{\{y\}} (Y_k)$.
Similarly, we define the sequences of  waiting times. For all $y\in\BbZ$ define $S_0(y):=-1$ and recursively for $k\geq 1$: $S_k(y)=\inf\{l>S_{k-1}(y): Y_l=y\}$. The random variables 
$\xi_k^{(y)}= T_{S_k(y)+1}-T_{S_k(y)}-1$ represent then the waiting time at level $y$ during the $k^\textrm{th}$ visit at that level. 
Due to strong Markov property, the doubly infinite sequence $(\xi_k^{(y)})_{y\in\BbZ, k\in\BbN^*}$ are independent $\BbN$-valued random variables with geometric distribution of parameter $p=1/3$; $q$ always stands for $1-p$ in the sequel.
\begin{defi}{}
Suppose the vertical 
skeleton and the environments of the orientations
are given. Let $(\xi^{(y)}_n)_{n\in\BbN, y\in\BbZ}$ be the previously defined doubly infinite sequence
of
geometric random variables of parameter $p=1/3$  and $\eta_n(y)$ the occupation measures of the vertical skeleton. We call 
\souligner{horizontally  embedded} random walk the process $(X_n)_{n\in\BbN}$ with
\[X_n  
=  \sum_{y\in\BbZ} \varepsilon_y \sum_{i=1}^{\eta_{n-1}(y)} \xi_i^{(y)},
\ \ n\in\BbN.\]
\end{defi}

\begin{lemm}{(See \cite[lemma 2.7]{CamPet-rwrol}).}
Let $T_n=n+\sum_{y\in\BbZ} \sum_{i=1}^{\eta_{n-1}(y)} \xi_i^{(y)}$
be the instant just after the random walk $(\bM_k)$ has performed
its $n^{\textrm{th}}$ vertical move (with the convention that the
sum $\sum_{i}$ vanishes whenever $\eta_{n-1}(y)=0$). Then
$\bM_{T_n}=(X_n,Y_n)$.
\end{lemm}

Define $\sigma_0=0$ and recursively, for $n=1,2,\ldots$, 
$\sigma_n=\inf\{k>\sigma_{n-1}: Y_k=0\}>\sigma_{n-1}$,
the $n^{\textrm{th}}$ return to the origin for the vertical
skeleton. Then obviously, $\bM_{T_{\sigma_n}}=(X_{\sigma_n},0)$.
To study the recurrence or the transience of $(\bM_k)$, we must study
how often $\bM_k=(0,0)$. Now, $\bM_{T_k}=(0,0)$ if and only if $X_k=0$ and
$Y_k=0$. Since $(Y_k)$ is a simple random walk, the event
$\{Y_k=0\}$ is realised only at the instants $\sigma_n$, $n=0,1,2,\ldots$.

\begin{remn} The significance of the random variable $X_n$ is the
horizontal displacement after $n$ vertical moves of the skeleton $(Y_l)$.
Notice that the random walk $(X_n)$ has unbounded (although integrable)
increments. As a matter of fact, they are signed integer-valued geometric
random variables. Contrary to $(X_n)$, the increments of the process 
$(X_{\sigma_n})_{n\in\BbN}$, sampled at instants $\sigma_n$, are unbounded with heavy-tails.
 \end{remn}
 
 Recall that all random variables are defined on the same probability
space $(\Omega,\cA,\BbP)$; introduce the following sub-$\sigma$-algebras:
$\cH = \sigma(\xi_i^{(y)}; i\in\BbN^*, y\in\BbZ)$,
$\cG = \sigma(\rho_y, \lambda_y;  y\in\BbZ)$, and 
$\cF_n = \sigma(\psi_i; i=1,\ldots, n)$,
with $\cF\equiv\cF_\infty$.

\begin{lemm}{(See \cite[lemma 2.8]{CamPet-rwrol})}
\label{lem-return-of-M}
\[\sum_{l=0}^\infty \BbP(\bM_l=(0,0)| \cF\vee \cG)=
\sum_{n=0}^\infty \BbP(I(X_{\sigma_n},\varepsilon_0 \xi_0^0) \ni 0 | \cF\vee \cG),\]
where, $\xi_0^0$ has the same law as $\xi_1^{(0)}$ and, for $x\in\BbZ$,  $z\in \BbN$, and $\varepsilon=\pm1$,
$I(x,\varepsilon z)=\{x, \ldots, x+z\}$ if $\varepsilon =+1$ and
$\{x-z, \ldots, x\}$ if $\varepsilon=-1$.
\end{lemm}

\begin{lemm}{(See \cite[lemma 2.9]{CamPet-rwrol})}
\label{lem-equivalence-MX}
\begin{enumeratec}
\item
If $\sum_{n=0}^\infty \BbP_0(X_{\sigma_n}=0|\cF\vee \cG)=\infty$ then 
$\sum_{l=0}^\infty \BbP(\bM_l=(0,0)| \cF\vee \cG)=\infty$.
\item
If $(X_{\sigma_n})_{n\in\BbN}$ is transient then $(M_n)_{n\in\BbN}$ is also transient.
\end{enumeratec}
\end{lemm}

Let $\xi$ be a geometric random variable equidistributed with $\xi_i^{(y)}$.
Denote 
\[\chi(\theta)=\BbE \exp(i\theta\xi)=\frac{q}{1-p\exp(i\theta)}
=r(\theta)\exp(i\alpha(\theta)), \ \ \theta
\in[-\pi,\pi]\]
its characteristic function, where
\[r(\theta)=|\chi(\theta)|=\frac{q}{\sqrt{q^2+2p(1-\cos\theta)}}=r(-\theta); \ 
\alpha(\theta)=\arctan\frac{p\sin\theta}{1-p\cos\theta}=-\alpha(-\theta).\]
Notice that $r(\theta)<1$ for $\theta\in [-\pi,\pi]\setminus\{0\}$.
 Then
\begin{eqnarray*}
\BbE\exp(i\theta X_n) &=& \BbE\left(\BbE(\exp(i\theta X_n)|\cF\vee\cG)\right)
\; =\; \BbE\left(\BbE(\exp(i\theta\sum_{y\in\BbZ}\varepsilon_y 
\sum_{i=1}^{\eta_{n-1}(y)}\xi_i^{(y)}|\cF\vee\cG)\right)\\
&=&\BbE\left(\prod_{y\in\BbZ}\chi(\theta\varepsilon_y)^{\eta_{n-1}(y)}\right).
\end{eqnarray*}

\section{Proof of transience}

Introduce, as was the case in \cite{CamPet-rwrol}, constants $\delta_i>0$ for $i=1,2,3$ and for $n\in\BbN$ the sequence of events $A_n=A_{n,1}\cap A_{n,2}$ and $B_n$ defined by
\[
A_{n,1} = \left\{\omega\in\Omega: 
\max_{0\leq k\leq 2n}|Y_k|<n^{\frac{1}{2}+\delta_1}\right\}; \; 
A_{n,2} = \left\{\omega\in\Omega: \max_{y\in\BbZ}\eta_{2n-1}(y)<
n^{\frac{1}{2}+\delta_2}\right\},\]
\[B_n =\left \{\omega\in A_n: 
\left|\sum_{y\in\BbZ}\varepsilon_y \eta_{2n-1}(y)\right|
>n^{\frac{1}{2}+\delta_3}\right\};\]
the range of possible values for $\delta_i$, $i=1,2,3$, will be chosen
later (see the end of the proof of proposition \ref{prop-estim-Bnc}). 
Obviously $A_{n,1}, A_{n,2}$ and hence $A_n$ belong to $\cF_{2n}$;
moreover 
$B_n\subseteq A_n$ and $B_n\in \cF_{2n}\vee\cG$. We denote in the sequel generically 
$d_{n,i}=n^{\frac{1}{2}+\delta_i}$, for $i=1,2,3$.

Since $B_n\subseteq A_n$ and both sets are $\cF_{2n}\vee\cG$-measurable,
decomposing the unity as
$1=\id_{B_n}+\id_{A_n\setminus B_n}+ \id_{A_n^c}$,
we get
$p_n=p_{n,1}+p_{n,2}+p_{n,3}$,
where
$p_n= \BbP(X_{2n}=0; Y_{2n}=0)$,
$p_{n,1}=  \BbP(X_{2n}=0; Y_{2n}=0;B_n)$,
$p_{n,2}=\BbP(X_{2n}=0; Y_{2n}=0;A_n\setminus B_n)$, and 
$p_{n,3}=  \BbP(X_{2n}=	0; Y_{2n}=0;A_n^c)$.
By repeating verbatim the reasoning  of propositions 4.1 and 4.3 of \cite{CamPet-rwrol}, we get 
\begin{prop}
\label{prop-estimate-pn13}
For large $n$, there exist  $\delta>0$ and $\delta'>0$ and $c>0$ and $c'>0$ such that
\[p_{n,1}=\cO(\exp(-cn^\delta)) \ \ \textrm{and}\ \ p_{n,3}=\cO(\exp(-c'n^{\delta'})).\]
\end{prop}
Consequently $\sum_{n\in\BbN}( p_{n,1}+ p_{n,3})<\infty$.
The proof will be complete if we show that $\sum_{n\in\BbN} p_{n,2}<\infty$.

\begin{lemm}
\label{lem-estim-conditional}
On the set $A_n\setminus B_n$, we have --- uniformly on $\cF\vee \cG$ ---
\[\BbP(X_{2n}=0|\cF\vee\cG)=\cO(\sqrt{\frac{\ln n}{n}}).\]
\end{lemm} 
\begin{proof}
Use the $\cF\vee\cG$-measurability of the variables $(\varepsilon_y)_{y\in\BbZ}$ and
$(\eta_n(y))_{y\in\BbZ,n\in\BbN}$ to express the conditional characteristic function
of the variable $X_{2n}$ as follows:
\[\chi_1(\theta)=\BbE(\exp(i\theta X_{2n})|\cF\vee\cG)=\prod_{y\in\BbZ}\chi(\theta\varepsilon_y)^
{\eta_{2n-1}(y)}.\]
Hence, $\BbP(X_{2n}=0|\cF\vee\cG)=\frac{1}{2\pi}\int_{-\pi}^{\pi}
\chi_1(\theta)d\theta.$
Now use the decomposition of $\chi$ into a  the modulus part, $r(\theta)$ --- that is an even
function of $\theta$ ---  and the angular part of $\alpha(\theta)$ 
and the fact that there is a constant $K<1$ such that
for $\theta\in[-\pi,-\pi/2]\cup[\pi/2,\pi]$ we can bound $r(\theta)<K$ to majorise
\[\BbP(X_{2n}=0|\cF\vee\cG)\leq \frac{1}{\pi}\int_{0}^{\pi/2}
r(\theta)^{2n}d\theta +\cO(K^n).\]
Fix $a_n=\sqrt{\frac{\ln n}{n}}$ and split the integration integral $[0,\pi/2]$ into 
$[0,a_n]\cup[a_n,\pi/2]$. For the first part, we majorise the integrand by 1, so that
$\int_{0}^{a_n}      
r(\theta)^{2n}d\theta \leq a_n.$

For the second part, we use the fact that $r(\theta)$ is decreasing for $\theta\in [0,\pi/2]$. Hence
$\frac{1}{\pi}\int_{a_n}^{\pi/2}r(\theta)^{2n}d\theta \leq \frac{1}{2} r(a_n)^{2n}$. 
Now, $\lim_{n\rightarrow \infty} a_n=0$, hence for large $n$ it is enough to study the behaviour of $r$ near 0, namely $r(\theta)\asymp 1-\frac{3}{8}\theta^2 +\cO(\theta^4)$. It follows that
$r(a_n)^{2n}\asymp (1-\frac{3}{4} \frac{\ln n}{2n})^{2n}\asymp\exp(-\frac{3}{4} \ln n)=n^{-\frac{3}{4}}$.
Since the estimate of the first part dominates, the result follows.
\end{proof}


\begin{lemm}{}
\label{lem:anderson}
Let $d$ be a positive  integer,  $Z$  an integer-valued random variable with law $\mu_Z$, and $G$ a centred Gaussian random variable with variance $d^2$, but otherwise independent of $Z$. Then, there exists a constant $C>0$ (independent of $d$ and of the law of $Z$) such that 
\[\BbP(|Z|\leq \frac{d}{2}) \leq C \BbP(|Z+G|\leq d).\]
\end{lemm}

\begin{proof}
Denote by $\gamma(g)= \frac{1}{\sqrt{2\pi d^2}} \exp(-\frac{g^2}{2 d^2})$ the density of the Gaussian random variable $G$ and observe that on $[-\frac{d}{2},\frac{d}{2}]$ the density is minorised by $\gamma(g)\geq \frac{2C^{-1}}{d}$ with $C=\sqrt{2\pi e}$. Then 
\[
\BbP(|Z+G|\leq d) \geq  \int_{-\frac{d}{2}}^{\frac{d}{2}} \mu_Z([-d-g, \ldots, d-g]) \gamma(g)dg
\geq 2C^{-1} \mu_Z(\{-\frac{d}{2}, \ldots, \frac{d}{2}\}).
\]
\end{proof}

\begin{prop}
\label{prop-estim-Bnc}
For all $\beta<1$, there exists a $\delta_\beta>0$ such that --- uniformly in $\cF$ --- for all large $n$
\[\BbP(A_n\setminus B_n|\cF)=\cO(n^{-\delta_\beta}).\]
\end{prop}
\begin{proof}
The required probability is an estimate, on the event $A_n$, of
the conditional probability $\BbP(|\sum_{y\in\BbZ} \zeta_{y,n}|\leq d_{n,3}|\cF)$,
where we denote $\zeta_{y,n}=\varepsilon_y \eta_{2n-1}(y)$. Extend the probability space
$(\Omega,\cA,\BbP)$ to carry an auxilliary variable
$G$ assumed to be centred Gaussian with variance $d_{n,3}^2$, (conditionally on $\cF$) 
independent of the $\zeta_{y,n}$'s. Since  $G$
is a symmetric random variable and $[-d_{n,3},d_{n,3}]$
is a symmetric set around 0, then by lemma \ref{lem:anderson},
there exists a positive constant
$c:=\sqrt{2\pi e}$ (hence independent of $n$) such that
\[\BbP(|\sum_{y\in\BbZ}\zeta_{y,n}|\leq d_{n,3}|\cF)\leq
c \BbP(|\sum_{y\in\BbZ}\zeta_{y,n}+G|\leq d_{n,3}|\cF).\]
Let $\chi_2(t)=\BbE(\exp(i t \sum_y\zeta_{y,n})|\cF)=\prod_y A_{y,n}(t),$
where $A_{y,n}(t)= \BbE\left(\exp(i t\zeta_{y,n} |\cF\right)$, 
and $\chi_3(t)=\BbE(\exp(i t G)|\cF)=\exp(-t^2d^2_{n,3}/2).$
Therefore, $\BbE(\exp(i t (\sum_y\zeta_{y,n}+G))|\cF)=\chi_2(t)\chi_3(t),$
and using the Plancherel's formula,
\[\BbP(|\sum_{y\in\BbZ}\zeta_{y,n}+G|\leq d_{n,3}|\cF)=\frac{d_{n,3}}{\pi}
\int_\BbR\frac{\sin(td_{n,3})}{td_{n,3}}\chi_2(t)\chi_3(t) dt\leq Cd_{n,3} I,\]
where $I=\int_\BbR|\chi_2(t)|\exp(-t^2d^2_{n,3}/2) dt.$
Fix $b_n=\frac{n^{\delta_4}}{d_{n,3}}$, for some $\delta_4>0$  and split the integral defining $I$ into
$I_1+I_2$, the first part being for $|t|\leq b_n$ and the second for
$|t|> b_n$.

We have
\begin{eqnarray*}
I_2&\leq& C \int_{|t|> b_n} \exp(-t^2 d^2_{n,3}/2) \frac{dt}{2\pi}\;
=\; \frac{C}{d_{n,3}} \int_{|s|> n^{\delta_4}} \exp(-s^2/2) \frac{ds}{2\pi}\\
&\leq&2 \frac{C}{d_{n,3}}\frac{1}{n^{\delta_4}} \frac{\exp(-n^{2\delta_4}/2)}{2\pi},
\end{eqnarray*}
because the probability that a centred normal random variable of variance 1,
whose density is denoted $\phi$,  exceeds a threshold
$x>0$ is majorised by $\frac{\phi(x)}{x}$.

For $I_1$ we get, $ I_1\leq \int_{|t|\leq b_n} \prod_y |A_{y,n}(t)| dt.$

Assume for the moment that the inequality $|t\eta_{2n-1}(y)|\leq 1$ holds. Use then the fact that $\cos(x)\leq 1-\frac{x^2}{4}$, valid for $|x|\leq 1$, to write 
\begin{eqnarray*}
|A_{y,n}(t)| &=& |\BbE_0(\exp(it\varepsilon_y \eta_{2n-1}(y))\cF)|\\
&=&|(1-\frac{c}{|y|^\beta}) \exp(it\eta_{2n-1}(y) f(y))+\frac{c}{|y|^\beta} \cos(t\eta_{2n-1}(y))|\\
&\leq& 1-\frac{c}{|y|^\beta}+ \frac{c}{|y|^\beta} (1-\frac{t^2\eta_{2n-1}^2(y)}{4})\\
&=& 1-\frac{ct^2\eta_{2n-1}^2(y)}{4|y|^\beta}\;
\leq\; \exp(-\frac{ct^2\eta_{2n-1}^2(y)}{4|y|^\beta}).
\end{eqnarray*}
The assumed inequality  $|t\eta_{2n-1}(y)|\leq 1$ is verified whenever the constants $\delta_2, \delta_3$ and $\delta_4$ are chosen so that $\delta_2+\delta_4-\delta_3<0$ holds, because $|t|\leq b_n$ and 
$b_n=\frac{n^{d_4}}{n^{1/2+\delta_3}}$, whereas $\eta_{2n-1}(y) \leq n^{1/2+\delta_2}$.
Therefore,
\[|\chi_2(t)|\leq \prod_y \exp\left(-\frac{t^2}{4} \eta_{2n-1}^2(y) \frac{c}{|y|^\beta}\right).\]
 Now, define $\pi_n(y)=\frac{\eta_{2n-1}(y)}{2n}$; obviously $\sum_y \pi_n(y)=1$, establishing that $(\pi_n(y))_y$ is a probability measure on $\BbZ$. Therefore, applying H\"older's
inequality we obtain $I_1\leq \prod'_y J_n(y) ^{\pi_n(y)}$, where $\prod'_y$ means that the product runs over those $y$ such that $\eta_{2n-1}(y)\not =0$ and
\begin{eqnarray*}
 J_n(y)&=&\int_{- b_n}^{b_n}\exp\left(-\frac{t^2}{4} \eta_{2n-1}^2(y) \frac{c}{|y|^\beta}\frac{1}{\pi_n(y)}\right)dt\\
&=&\sqrt{\frac{2\pi |y|^\beta}{cn\eta_{2n-1}(y)}}
\int_{- b_n \sqrt{\frac{cn\eta_{2n-1}(y)}    { |y|^\beta}}} ^{ b_n \sqrt{\frac{cn\eta_{2n-1}(y)}    { |y|^\beta}}}\exp(-v^2/2) \frac{dv}{\sqrt{2\pi}}\\
&\leq& \sqrt{\frac{4\pi}{c}} \exp\left(-\log 2n -\frac{1}{2} \log\pi_n(y)+\frac{\beta}{2} \log |y|\right).
\end{eqnarray*}
We conclude that
\[
I_1 \leq \prod_y \mbox{}^{'}J_n(y) ^{\pi_n(y)}
\leq \sqrt{\frac{2\pi}{c}}\exp \left(-\log 2n +\frac{1}{2} H(\pi_n)+ \frac{\beta}{2}
\sum_{y} \pi_n(y) \log|y|\right)\]
and $H(\pi_n)$ is the entropy of the probability measure $\pi_n$, reading (with the convention $0\log0=0$)
\[H(\pi_n):=-\sum_y \pi_n(y) \log \pi_n(y)\leq \log \card C_n,\] where
$C_n:=\supp \pi_n$ and, on $A_n$,  $\card C_n\leq 2n^{\frac{1}{2}+\delta_1}$. 
We conclude that  we can always chose the parameters $\delta_1$ and $\delta_3$ such that, for every $\beta<1$ there exists a parameter $\delta_\beta>0$ such that 
$d_{n,3} I_1 \leq Cn^{-\delta_\beta}.$
\end{proof}

\begin{coro}
\[\sum_{n\in\BbN}p_{n,2}<\infty.\]
\end{coro}
\begin{proof}
Recall that for the standard random walk 
$\BbP(Y_{2n}=0)=\cO(n^{-1/2})$;  combining with the estimates obtained in
\ref{lem-estim-conditional} and \ref{prop-estim-Bnc}, we have 
\begin{eqnarray*}
p_{n,2}&=&\BbP(X_{2n}=0; Y_{2n}=0;A_n\setminus B_n)\\
&=& \BbE(\BbE\left(\id_{Y_{2n}=0}\left[
\BbE(\id_{A_n\setminus B_n} \BbP(X_{2n}=0|\cF\vee\cG) | \cF)\right]\right))\\
&=& \cO( n^{-1/2} n^{-\delta_\beta} \sqrt{\frac{\ln n}{n}})\;
=\; \cO(n^{-(1+\delta_\beta)} \sqrt{\ln n}),
\end{eqnarray*}
proving thus the summability of $p_{n,2}$.
\end{proof}

\noindent
\textit{Proof of the statement on transience of the theorem \ref{thm:main}:}
$p_n=p_{n,1}+p_{n,2}+p_{n,3}$ is summable because all the partial probabilities
$p_{n,i}$, for $i=1,2,3$ are all shown to be summable.
\eproof


\section{Proof of recurrence}

We define additionally the following sequence of random times:
\[\tau_0\equiv0 \ \ \textrm{and}\ \ \tau_{n+1}= \inf\{k:k>\tau_n, |Y_k-Y_{\tau_n}|=Q\} \ \ \textrm{for}\ \ n\geq 0.\]
The random variables $(\tau_{n+1}-\tau_n)_{n\geq 0}$ are independent and for all $n$ the variable  $\tau_{n+1}-\tau_n$ has the same distribution (under
$\BbP_0$) as $\tau_1$.
It is easy  to show  further (see proposition 1.13.4 of the textbook \cite{BhattacharyaWaymire2007} for instance) that these random variables have exponential moments i.e.\ $\BbE_0(\exp(\alpha \tau_1))<\infty$ for $|\alpha|$ sufficiently small.

Let $\BbZ_Q=\BbZ/Q\BbZ=\{0,1,\ldots, Q-1\}$ with integer addition replaced by addition modulo $Q$ and for any $y\in\BbZ$  denote by $\overline{y}=y \mod Q\in\BbZ_Q$. Consistently, we define $\overline{Y}_n= Y_n\mod Q$.

\begin{lemm}
\label{lem:N}
Define for $n\geq 1$ and $\overline{y}\in\BbZ_Q$, 
\[N_n(\overline{y}):=  \overline{\eta}_{\tau_{n-1},\tau_n-1}(\overline{y})=\sum_{k=\tau_{n-1}}^{\tau_n -1} \id _{\overline{y}}(\overline{Y}_k).\]
Then, for every $\overline{y}\in\BbZ_Q$,
\begin{enumeratec}
\item the conditional laws of $N_1(\overline{y})$ with respect to the events $\{Y_{\tau_1}=Q\}$ and  $\{Y_{\tau_1}=-Q\}$ are the same, and
\item the following equalities hold
\[\BbE_0 N_1(\overline{y})=\BbE_0\left( N_1(\overline{y})\;\middle|\;Y_{\tau_1}=Q\right)=\BbE_0 \left( N_1(\overline{y})\;\middle|\;Y_{\tau_1}=-Q\right) =\frac{\BbE_0\tau_1}{Q}.\] 
\end{enumeratec}
\end{lemm}

\begin{proof}
\begin{enumeratec}
\item
Denote by $U:=\{Y_{\tau_1}=Q\}$ and $D:=\{Y_{\tau_1}=-Q\}$ the sets of conditioning.  
Assume that $Y$ is a trajectory in $U$ and define $R:=\max\{t: 0\leq t <\tau_{1}, Y_t=0\}$.
Now, 
between times $0$ and $R$, the path $Y$ wanders around the level $0$. For times $t$ such that $R<t<\tau_{1}$, the path remains strictly confined within the (interior of the) strip.  

For any trajectory $Y$ in $U$, we shall define a new trajectory $V$ --- bijectively determined from $Y$ ---  belonging to $D$ as follows:
\[
V_t= \left\{\begin{array}{lcl}
 Y_{t} & \textrm{for} & 0 \leq t \leq R\\
Y_{\tau_{1}-(t-R)}-Q& \textrm{for} & R\leq t \leq  \tau_{1}.
\end{array}\right.
\]
Obviously, the above construction is a bijection. Hence for trajectories not in $U$ (i.e.\ trajectories in $D$) the modified trajectory is defined inverting the previous transformation. 
The figure \ref{fig:modified-reflection-principle} illustrates the construction (modified reflection principle).
\begin{figure}[h]
\begin{center}
\begin{tikzpicture}[scale=0.2]
\foreach \j in {0, ..., 15} \draw [-, gray!30, very thin] (\j, -6) -- (\j, 6);
 \foreach \j in {-6, ..., 6} \draw [-, gray!30, very thin] (0, \j) -- (16, \j);
 \foreach \j in {-6, 0, 6} \draw [-, thick] (0, \j) -- (16, \j);
\draw [-,red]  (0,0) -- (1,1) -- (2,0) -- (3,-1) -- (4,0) -- (5,1) -- (6,0) -- (7,-1) -- (8,0) -- (9, 1) -- (10,2) -- (11,1) -- (12, 2) -- (13, 3) -- (14,4) -- (15,5) -- (16,6);
\draw (0,0) node [anchor=north] {$0$};
\draw (16,0) node [anchor=north] {$\tau_{1}$};
\draw [-, thin, green] (8, -6) -- (8, 6);
\draw [<-, very thin] (0,3) -- (3, 3);
\draw [->, very thin] (5,3) -- (8, 3);
\draw (4, 2) node [anchor=south] {$R$};
\draw (-0.5,-6) node [anchor=east] {$-Q$};
\draw (-0.5,6) node [anchor=east] {$+Q$};
\draw (-0.5,0) node [anchor=east] {$0$};
\end{tikzpicture}
\hskip1cm
\begin{tikzpicture}[scale=0.2]
\foreach \j in {0, ..., 15} \draw [-, gray!30, very thin] (\j, -6) -- (\j, 6);
 \foreach \j in {-6, ..., 6} \draw [-, gray!30, very thin] (0, \j) -- (16, \j);
 \foreach \j in {-6, 0, 6} \draw [-, thick] (0, \j) -- (16, \j);
\draw [-,red]  (0,0) -- (1,1) -- (2,0) -- (3,-1) -- (4,0) -- (5,1) -- (6,0) -- (7,-1) -- (8,0);
\draw [-,blue] (8,0) -- (9, -1) -- (10,-2) -- (11,-3) -- (12, -4) -- (13, -5) -- (14,-4) -- (15,-5) -- (16,-6);
\draw (0,0) node [anchor=south] {$0$};
\draw (16,0) node [anchor=south] {$\tau_{1}$};
\draw [-, thin, green] (8, -6) -- (8, 6);
\draw [<-, very thin] (0,3) -- (3, 3);
\draw [->, very thin] (5,3) -- (8, 3);
\draw (4, 2) node [anchor=south] {$R$};
\draw (-0.5,-6) node [anchor=east] {$-Q$};
\draw (-0.5,6) node [anchor=east] {$+Q$};
\draw (-0.5,0) node [anchor=east] {$0$};
\end{tikzpicture}
\end{center}
\caption{
\label{fig:modified-reflection-principle}
Illustration of the modified reflexion principle. The left figure depicts a detail of the up-crossing excursion, occurring between times $0$ and $\tau_{1}$. The left figure depicts the  details of a new admissible path bijectively  obtained from $Y$  by defining it as identical to $Y$  for the times $0\leq t \leq R$ and then  by reverting the flow of time and displacing the remaining portion of the path by $-Q$, as explained in the text.}
\end{figure}
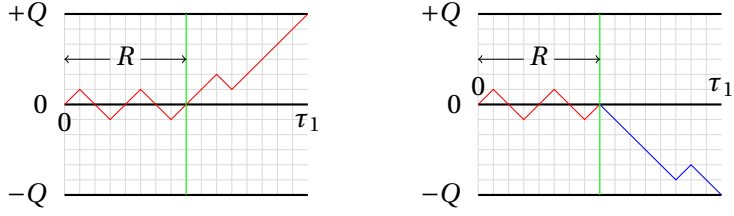

On denoting by $\eta$ the occupation measure of the process $Y$ and  by $\kappa$ the one of $V$, 
we have
$\overline{\kappa}_{\tau_{1}-1}(\overline{y}):=\sum_{i=0}^{\tau_1-1} \id_{\{\overline{y}\}}(\overline{V}_i)=
\sum_{i=0}^{\tau_1-1} \id_{\{\overline{y}\}}(\overline{Y}_i)=:\overline{\eta}_{\tau_{1}-1}(\overline{y})$
by construction of the path $V$. This remark implies that $\overline{\eta}_{\tau_{1}-1}(\cdot)$ and 
$\overline{\kappa}_{\tau_{1}-1}(\cdot)$ have the same law.

\item
Since the random walk $(Y_n)$ is symmetric, the probability of exiting the strip of width $Q$ by up-crossing is the  same as for a down-crossing. 
Hence 
$\BbE_0 N_1(\overline{y})= \frac{1}{2} \BbE_0\left( N_1(\overline{y})\;\middle|\;Y_{\tau_1}=Q\right)+\frac{1}{2} \BbE_0 \left( N_1(\overline{y})\;\middle|\;Y_{\tau_1}=-Q\right)$. This remark, combined with the equality of conditional laws established in 1.\    establishes the leftmost and the central equalities of the statement.

 To prove the rightmost equality, 
let $g:\BbZ\rightarrow \BbR$ be a bounded function and denote by
$S_n[g]= \sum_{k=0}^{n-1} g(Y_k)$. On defining $W_n[g]= \sum_{k=\tau_n}^{\tau_{n+1}-1} g(Y_k)$ and $R_n=\max\{k: \tau_k\leq n\}$, we have the decomposition:
\begin{equation*}
S_n[g]=  \sum_{k=0}^{R_n} W_k[g]-\sum_{k=n}^{\tau_{R_n+1}-1} g(Y_k). 
\end{equation*}

Since $\tau_{R_n+1}-n \leq \tau_{R_n+1}-\tau_{R_n}$, we have, thanks to the boundedness of $g$, that $\frac{1}{n} \left| \sum_{k=n}^{\tau_{R_n+1}-1} g(Y_k)\right| \leq \frac{\tau_{R_{n+1}}-\tau_{R_n}}{n} \sup_{y\in\BbZ} |g(z)|$. 
Since 
\begin{eqnarray*}
\BbP(\tau_{R_{n+1}}-\tau_{R_n}=l) &\leq& \sum_{k=0}^n\BbP(\tau_{{k+1}}-\tau_{k}=l;R_n=k)\\
&\leq&  \sum_{k=0}^n\BbP(\tau_{k+1}-\tau_{k}=l),
\end{eqnarray*}
it follows that for all $\varepsilon >0$, we have $\sum_{k=1}^n\BbP(\tau_{k+1}-\tau_k\geq \varepsilon n)\leq n \BbP(\tau_1\geq \varepsilon n)$  which tends to 0, when $n\rightarrow\infty$, thanks to Markov inequality and the existence of exponential moments for $\tau_1$.

It remains to estimate $\frac{S_n[g]}{n}$ by $\frac{R_n}{n}\frac{1}{R_n}\sum_{k=0}^{R_n} W_k[g]$. Obviously $R_n\rightarrow \infty$ a.s., as $n\rightarrow \infty$, and, by the renewal theorem (see p.\ 221 of \cite{BhattacharyaWaymire2007} for instance), $\frac{R_n}{n}\rightarrow \frac{1}{\BbE_0 \tau_1}$ a.s.
Fix any $\overline{y}\in\BbZ_Q$ and choose $g(z):=\id_{\{\overline{y}\}} (z \mod Q)$. For this $g$, we have $S_n[g]=\overline{\eta}_n(\overline{y})$, where $\overline{\eta} _n(\overline{y})= \sum_{k=0}^{n-1} 
\id_{\{\overline{y}\}} (\overline{Y}_k)$. But $(\overline{Y}_k)$ is a simple random walk on the finite set $\BbZ_Q$ therefore admits a unique invariant probability $\overline{\pi}(\overline{y})= \frac{1}{Q}$. By the ergodic theorem for Markov chains, we have $\frac{S_n[g]}{n}\rightarrow \frac{1}{Q}$ a.s.

Additionally, for this choice of $g$, the sequence $(W_k[g])_{k\in\BbN}$ are independent random variables, identically distributed as $N_1(\overline{y})$. We conclude by applying the law of large numbers  to the ratio $\frac{1}{R_n}\sum_{k=1}^{R_n} W_k[g]$.
\end{enumeratec}
\end{proof}


To prove almost sure recurrence, it is enough to  show $\sum_{k\in\BbN} \BbP_0\left(X_{\sigma_k}=0, Y_{\sigma_k}=0\;\middle|\; \cG\right)=\infty$ a.s. 
If $\beta >1$ then  $\sum_y \BbP(\lambda_y=1)<\infty$; hence,  by Borel-Cantelli lemma,  there is almost surely a finite number of $y$'s such that $\lambda_y=1$, i.e.\ the $\cG$-measurable random variable $l(\omega)=\max \{|y|: \lambda_y=1\}/Q$ is almost surely finite. Fix an integer $L(\omega)\geq l(\omega)+1$, and introduce the $\cF\vee \cG$-measurable random sets:
\begin{eqnarray*}
F_{L,2n}(\omega) &=& \left\{k: 0\leq k \leq 2n-1; |Y_{\tau_{k}(\omega)}(\omega)| \leq L(\omega)Q;   |Y_{\tau_{k+1}(\omega)}(\omega)| \leq L(\omega)Q\right\}\\
G_{L,2n}(\omega) &=& \left\{k: 0\leq k \leq 2n-1; |Y_{\tau_{k}(\omega)}(\omega)| \geq L(\omega)Q;   |Y_{\tau_{k+1}(\omega)}(\omega)| \geq L(\omega)Q\right\}.
\end{eqnarray*}
To simplify notation, we drop  explicit reference to the $\omega$ dependence of those sets.


Denote by $\textsf{Adm}(2n)$ the set of \souligner{admissible paths}   
$\bz=(z_0, z_1, \ldots, z_{2n-1}, z_{2n})\in \BbZ^{2n+1}$  satisfying 
$ |z_{i+1}-z_i|=1$ for $ i=0,\ldots 2n-1$ and $z_0=0$.
For any $\bz\in \textsf{Adm}(2n)$, we denote  $C[\bz]$ the  cylinder set
\[C[\bz]= \left\{\omega\in\Omega:Y_0(\omega)=Qz_0=0, Y_{\tau_1(\omega)}(\omega)= Qz_1, \ldots, Y_{\tau_{2n}(\omega)}(\omega)=Qz_{2n}\right\}\in\cF.\]

Denote by $\theta_k= X_{\tau_{k+1}}-X_{\tau_k}$, for $k\in\{0, \ldots, 2n-1\}$, and observe that 
\[X_{\tau_{2n}}=\sum_{k=0} ^{2n-1} \theta_k= \sum_{k\in F_{L,2n}} \theta_k+
 \sum_{k\in G_{L,2n}} \theta_k,\]
the sums appearing in the above decomposition referring to disjoint excursions.   

\begin{prop}
\label{pro:vanishing-a}
For every $\bz\in \textsf{Adm}(2n)$ and every $k\in  G_{{L,2n}}(\omega)$, with $\omega\in C[\bz]$, 
\[a_k:=\BbE_0\left(\theta_k\middle|C[\bz]; \cG\right)=0.\]
\end{prop} 
\begin{proof}
Let $\bz$ be an arbitrary admissible path and suppose that 
$k$ corresponds, say,  to an up-crossing (\textit{i.e.}\ $z_{k+1}-z_k=1$) and abbreviate to $z:=z_k$ and $z+1= z_{k+1}$ in  order to simplify notation.  
Since $\bz\in \textsf{Adm}(2n)$, then for all $\omega\in C[\bz]$, the  random times $\tau_1,\ldots, \tau_{2n}$ are \textbf{compatible} with $\bz$, meaning, in particular, that
\[
Y_{\tau_k}=Qz \ \ \textrm{and}\ \ Y_{\tau_{k+1}}  = Qz+Q.
\]

The horizontal increments $(\theta_k)_{k\in G_{L,2n}}$,  conditionally on $C[\bz]$ and $\cG$,  are independent. To simplify notation, introduce the symbol $T:=\tau_{k+1}-\tau_k$; obviously, $T\elaw \tau_1$ conditionally on the starting point; more precisely $\BbP_{Qz_k}(T=t)=\BbP_0(\tau_1=t)$, for all $t\in\BbN$.

We are now in position to complete the proof of the proposition.
\begin{eqnarray*}
a_k&=&
\BbE_0\left(\sum_{y\in\BbZ} f(y)\sum_{i=0}^
{\eta_{\tau_k, \tau_{k+1}-1}(y)} \xi _i ^y 
 \middle | C[\bz]; \cG\right)\\
 &=& 
 \BbE(\xi_0^0)\sum_{y=Qz-Q+1}^{Qz+Q-1} f(\overline{y}) 
 \BbE_{Qz}\left (\eta_{T-1}(y)\middle | Y_{T} =Qz+Q;  C[\bz]; \cG\right) \BbP_{Qz}(Y_{T} =Qz+Q|C[\bz];\cG)
 \\
 &=&
   \BbE(\xi_0^0) \sum_{\overline{y}\in\BbZ_Q} f(\overline{y})
 \BbE_0\left (N_1(\overline{y})\middle | Y_{\tau_1} =Q\right)\\
 &=& \sum_{\overline{y}\in \BbZ_Q} f(\overline{y})  \frac{\BbE_0(\tau_1)}{Q}\BbE(\xi_0^0)\\
 &=&0,
\end{eqnarray*}
where we used  strong Markov property, lemma \ref{lem:N}, and the centring condition $\sum_{\overline{y}\in\BbZ_Q}
f(\overline{y})=0$ to conclude.
\end{proof}
The sampled process $Z_k=\frac{Y_{\tau_k}}{Q}\in\BbZ$ is a standard simple symmetric nearest neighbour random walk on $\BbZ$. For $z\in\BbZ$,  define the occupation measure $\varpi_n(z):=\varpi(\{z\})= \sum_{k=1}^n \id_{\{z\}}(Z_k)$.  
\begin{lemm}
\label{lem:varpi}
Fix $K>0$. For every $\delta>0$ there exists a constant $c>0$ such that for all $n$ sufficiently large,
\[\aP_n=\BbP_0\left(\max_{z:|z|\leq K}\varpi_{2n}(z) > c \sqrt{n}\; \middle \vert\; Z_{2n}=0\right) <\delta.\] 
\end{lemm}
\begin{remn}This lemma will be used in the course of the proof by fixing a $\cG$-measurable almost surely finite $K$, while $n$ will tend to infinity. Of course $c=c(K,\delta)$ depends on the choice of $K$ and $\delta$.
\end{remn}
\begin{proof} 
Denote by $m_n =\lfloor c\sqrt{n}\rfloor$.
Then, by conditional Markov inequality,
\begin{eqnarray*}
\aP_n 
\leq \sum_{z=-K}^K \BbP_0\left(\varpi_{2n}(z) > m_n\; \middle \vert\; Z_{2n}=0\right)
\leq  \sum_{z=-K}^K \frac{\BbE_0\left(\varpi_{2n}(z)\; \middle \vert\; Z_{2n}=0\right)}{m_n}.
\end{eqnarray*}
Now
\begin{eqnarray*}
\BbE_0\left(\varpi_{2n}(z) \; \middle \vert\; Z_{2n}=0\right)
&= & \sum_{k=1}^{2n} \BbP_0\left(Z_k=z\; \middle \vert\; Z_{2n}=0\right)\\
&=&  \sum_{k=1}^{2n}\frac{P^k(0,z)P^{2n-k}(z,0)}{P^{2n}(0,0)},
\end{eqnarray*}
where $P^l(0,z)=\BbP_0(Z_l=z)$. 
For all $z\geq0$ and all $l\geq z$, $P^l(0,z)=2^{-l}\binom{l}{\frac{l+z}{2}}$ if $l+z$ is even and 0 otherwise. We majorise $\binom{l}{\frac{l+z}{2}}\leq \binom{l}{\frac{l}{2}}$, when $l$ is even and
$\binom{l}{\frac{l+z}{2}}\leq \binom{l}{\frac{l-1}{2}}$ when $l$ is odd.
Using Stirling's formula, we see that for  all $l$ sufficiently large, the probability 
$P^l(0,z)$ is majorised independently of the parity of $l$  by a term equivalent (for large $l$) to $\frac{1}{\sqrt{l}}$. Consquently by choosing an appropriate constant $C$, the same majorisation holds for the remaining finite set of values of $l$. 
By approximating the sum by an integral, we get finally that 
\[\BbE_0(\varpi(z)|Z_{2n}=0)\leq C\int_0^{2n} \sqrt{\frac{2n}{t(2n-t)}}dt \leq e\sqrt{n}.\]
We conclude that $\aP_n\leq \delta$ provided that $c>\frac{2K+1}{C}$.
\end{proof}


\noindent\textit{Proof of the recurrence statement of theorem  \ref{thm:main}:}

We shall now fix $K=L$. For $\delta\in]0,1[$, let $c=c(K,\delta)$ be as in the previous lemma \ref{lem:varpi}.
From this very same lemma, we have  $\BbP_0\left(\card F_{L,2n} \leq c\sqrt{n}\right)\geq 1-\delta$
on the set $\{Z_{2n}=0\}$. Fix some constant $d$ and define 
\[\textsf{ConsAdm}(L, 2n,d)=\left\{\bz\in\textsf{Adm}(2n): z_{2n}=0; \vert\{k: 0\leq k <2n, |z_k| \leq L; |z_{k+1}\vert\leq L\}|\leq d\sqrt{n}\right\}\] 
the set of \souligner{constrained admissible paths}. 
(Here and in the sequel, we use indistinguishably the symbols  $|A|$ or $\card A$ to denote the cardinality of the discrete set $A$).
On the set $\{Z_{2n}=0\}$, obviously the equality $\{\card F_{L,2n}\leq d\sqrt{n}\}=\cup_{\bz\in\textsf{ConsAdm}(L, 2n, d)} C[\bz]$ holds. 
\begin{eqnarray*}
\BbP_0\left(X_{\tau_{2n}}=0; Y_{\tau_{2n}}=0 \; \middle \vert \; \cG\right)
&=&\BbP_0\left(X_{\tau_{2n}}=0; Z_{2n}=0 \; \middle \vert \; \cG\right)\\
&\geq& \BbP_0\left(X_{\tau_{2n}}=0; Y_{\tau_{2n}}=0; |F_{L,2n}| \leq d\sqrt{n} \; \middle \vert \;\{Z_{2n}=0\}; \cG\right)
\BbP_0(Z_{2n}=0)\\
&=&  \sum_{\bz\in\textsf{ConsAdm}(L, 2n,d)} 
\BbP_0\left(\{X_{\tau_{2n}}=0\}\cap C[\bz] \; \middle \vert \; \{Z_{2n}=0\}; \cG \right)\BbP_0(Z_{2n}=0)\\
&=&  \sum_{\bz\in\textsf{ConsAdm}(L, 2n,d)} 
\BbP_0\left(X_{\tau_{2n}}=0\; \middle \vert \; C[\bz];\cG\right) \BbP_0\left(C[\bz]\;\middle|\;\cG\right).
\end{eqnarray*}
 Now, for any $\bz\in \textsf{ConsAdm}(L, 2n,d)$,
\begin{eqnarray*}
\BbP_0\left(X_{\tau_{2n}}=0\; \middle \vert \; \cG, C[\bz]\right) 
&\geq &  \sum_{|m|\leq d\sqrt{n}}\BbP_0\left(
\sum _{k\in F_{L,2n}} \theta_k=m;  \sum _{k\in G_{L,2n}} \theta_k =-m\; \middle \vert \; \cG, C[\bz]\right)\\
&= &  \sum_{|m|\leq d\sqrt{n}} \BbP_0\left(
\sum _{k\in F_{L,2n}} \theta_k=m \; \middle \vert \; \cG, C[\bz]\right) 
 \BbP_0\left( \sum _{k\in G_{L,2n}} \theta_k =-m\; \middle \vert \;  \cG, C[\bz]\right).
\end{eqnarray*}
The joint probability  factors into the terms appearing in  the last line  because the $\cG$-measurable set-valued random variables $G_{L,2n}$ and $F_{L,2n}$ take disjoint values, hence the terms in $F_{L,2n}$ and $G_{L,2n}$ refer to different excursions of the random walk $Y$. Independence follows as a consequence of the strong Markov property.

By the proposition \ref{pro:vanishing-a}, we have $\BbE(\theta_k|C[\bz],\cG)=0$.
The  variables $(\theta_k)_{k\in G_{L,2n}}$ are independent and identically distributed conditionally to $\cG$ and $C[\bz]$; their common variance, $\sigma^2$, is finite because, 
\begin{eqnarray*}
\sigma^2 &=&\BbE_0(\theta_k^2|\cG, C[\bz])
\ = \ 
\BbE_{Qz_k} \left( \left[\sum _{y}\varepsilon_y \sum _{i=0} ^{\eta_{\tau_k,\tau_{k+1}-1}(y)} \xi^y_i\right]^2 \middle\vert\; \cG\right)\\
&\leq & \BbE_0(\tau_1) \BbE((\xi_0^0)^2) +\BbE_0(\tau_1^2)[\BbE(\xi_0^0)]^2+ 
[\BbE_0(\tau_1)]^2[\BbE (\xi_0^0)]^2
\ < \ \infty,
\end{eqnarray*}
where we have used  strong Markov property to bound the 
last term of the first line in the previous  formula by the second line.

 For $\bz\in\textsf{ConsAdm}(L,2n,d)$, we have further --- on $C[\bz]$ --- that  $2n- d\sqrt{n}\leq |G_{L,2n}|\leq 2n$. 
 Hence, for $|m|\leq d\sqrt{n}$, we can apply local limit theorem (see proposition 52.12, p.\ 706 of \cite{Port1994}  for instance),  reading   
\[\BbP_0\left(\sum_{k\in G_{L,2n}}\theta_k=-m\;\middle\vert\; \cG, C[\bz]\right)\geq \frac{c_{1}}{\sqrt{|G_{L,2n}|\sigma^2}} \exp\left(-\frac{m^2}{2|G_{L,2n}|\sigma^2}\right),\]
to obtain
$\BbP_0\left(\sum_{k\in G_{L,2n}}\theta_k=-m\; \middle \vert \;\cG, C[\bz]\right)\geq \frac{c_{2}}{\sqrt{n}}$, uniformly in $\bz$.
We can summarise the estimate obtained so far 
 \begin{eqnarray*}
 \BbP_0\left(X_{\tau_{2n}}=0, Y_{\tau_{2n}}=0\;\middle|\; \cG\right)&\geq& 
\frac{c_3}{\sqrt{n}}\sum_{\bz\in\textsf{ConsAdm}(L, 2n,d)}
\BbP_0(C[\bz]|\cG) 
\BbP_0\left(\left|\sum_{k\in F_{L,2n}}\theta_k\right|\leq d\sqrt{n}\;\middle |\;C[\bz];\cG\right).
\end{eqnarray*}
Now, $\{|\sum_{k\in F_{L,2n}} \theta_k|\leq d\sqrt{n}\}\supseteq 
\{\sum_{k\in F_{L,2n}} |\theta_k|\leq d\sqrt{n}\}\supseteq
\{\sum_{k\in F_{L,2n}} \Theta_k\leq d\sqrt{n}\}$, where
$\Theta_k =\sum_y\sum_{i=\eta_{\tau_k(y)}}^{ \eta_{\tau_{k+1}-1}(y)}\xi^y_i$ are i.i.d.\ conditionally on $C[\bz]$, with finite mean $0\leq \mu=\BbE \Theta_k=\BbE(\xi_0^0)\BbE(T_k)<\infty$ and variance $0\leq\sigma^2=\Var\Theta_k<\infty$,
where $T_k=\tau_{k+1}-\tau_k$ is the time needed for the vertical random walk to cross the strip bounded by $z_k$ and $z_{k+1}$.

Additionally, $\lim_{n\rightarrow\infty} |F_{L,2n}| =\infty$ a.s., due to the recurrence of the simple symmetric vertical random walk $(Y_k)$.
From the weak law of large numbers, it follows that for all $\varepsilon >0$ 
\[\lim_{n\rightarrow \infty} \BbP_0\left(\left|\frac{\sum_{k\in F_{L,2n}} \Theta_k}{|F_{L,2n}|}-\mu\right | \leq \varepsilon \right)=1,\]
hence  for all  $\alpha\in ]0,1[$ and sufficiently large $n$,  $\BbP_0(|\frac{\sum_{k\in F_{L,2n}} \Theta_k}{|F_{L,2n}|}-\mu| \leq \varepsilon )\geq \alpha$.
 Since, for $\bz\in \textsf{ConsAdm}(L, 2n,d)$ we have 
 $|F_{L,2n}|\leq d\sqrt{n}$, we conclude that  for  all $n$ sufficiently large, 
 $\BbP_0(|\sum_{k\in F_{L,2n}} \theta_k|\leq d'\sqrt{n}|\cG)>\alpha$, with any $d'>\mu d$.
 Finally, for $n$ sufficiently large, 
 \[\sum_{\bz\in\textsf{ConsAdm}(L, 2n,d)}
\BbP_0(C[\bz]|\cG)=\BbP_0\left(\sum_{z:|z|\leq L+1}\varpi(z)\leq d\sqrt{n}\;\middle|\;Z_{2n}=0\right)\BbP_0(Z_{2n}=0) \geq (1-\delta)\frac{c_3}{\sqrt{n}}\]
 from lemma \ref{lem:varpi}, where $Z_k=Y_{\tau_k}$.
 This concludes the proof of the recurrence.
\eproof

\noindent\textit{Proof of the proposition \ref{pro:deterministic}:}
Since $\|\bom{\lambda}\|<\infty$, it follows that there is a positive integer $l$ such that
$\lambda_y=0$ for all $y$ with $|y|>l$. Choosing then an integer $L> [l/Q]+1$ in the above proof of the recurrence part, we immediately conclude.
\eproof



\section{Conclusion, open problems, and further developments}

\label{sec:open}
As was apparent in the course of the proof of recurrence, the condition $\beta>1$ is used only to show that there are almost surely finitely many lines where the periodicity imposed by $f$ is perturbed by a random defect. Therefore this condition  can be improved. 
\cut{For instance, we can show that if the decay is of the form $\frac{c}{|y| \ln ^{\beta_1} |y|}$, with $\beta_1>1$ or  $\frac{c}{|y| \ln |y| \ln\ln ^{\beta_2} |y|}$, with $\beta_2>1$, etc., then the random walk is still recurrent.} As a matter of fact, the walk is recurrent provided that there exists an arbitrarily large integer $l$ such that the decay is of the form $c(|y|\ln|y|\cdots \ln_{l-1} |y|\ln_l^{\beta_l} |y|)^{-1}$ for some $\beta_l>1$ (arbitrarily close to 1), where $\ln_l$ is the $l$-times iterated logarithm.   Nevertheless, our methods do not allow the treatment of the really critical  case $\beta_0=1$. 

It is easy to build up examples in which the random walk is recurrent although there are infinitely many defects, provided they are sparse. The following deterministic construction illustrates this fact.
Let $(a_n)_{n\in\BbN}$ be an increasing sequence of positive numbers such that $a_n\rightarrow \infty$. For an arbitrary $\{0,1\}$-valued sequence $\bom{\lambda}=(\lambda_y)_{y\in\BbZ}$ we denote by  $\bom{\lambda}\rt{k}$ its  restriction to $\{-k,\ldots, k\}$ (meaning that $(\lambda\rt{k})_y=\lambda_y$ if $|y|\leq k$ and vanishes otherwise) for $k\in\BbN$; define also $\bom{\lambda}\rt{\infty}\equiv \bom{\lambda}$. Obviously $\|\bom{\lambda}\rt{k}\|=\card\{y\in\BbZ: |y|\leq k, \lambda_y=1\}$. For a given sequence $\bom{\lambda}$, we write $\BbP[\bom{\lambda}]_0(\cdot)$ for the probability measure corresponding to the environment $\bom{\lambda}$. 
We shall construct iteratively an infinite deterministic sequence $\bom{\lambda}=(\lambda_y)_{y\in\BbZ}$ of defects perturbing the periodic sequence determined by $f$ so that the corresponding random walk is recurrent.  The first defect is inserted at level 0, i.e.\ we initialise the sequence to $\lambda^{(1)}_y=\delta_{y,0}$ so that $\|\bom{\lambda}^{(1)}\|=1$.  Now the random walk with only one defect in the whole vertical axis is recurrent, meaning that 
\[\sum_{n\in\BbN} \BbP[\bom{\lambda}^{(1)}]_0(\bM_n=(0,0))=\infty.\] Therefore there exists a positive integer $L_1>0$ such that 
\[\sum_{n=1}^{L_1} \BbP[\bom{\lambda}^{(1)}]_0(\bM_n=(0,0))\geq a_1.\] But since in time $L_1$, the vertical random walk cannot be further than $L_1$ from 0, nothing changes if instead of choosing the sequence $\bom{\lambda}^{(1)}$ as above with $ \|\bom{\lambda}^{(1)}\|=1$, we chose \textit{any other} sequence $\bom{\lambda}'$ with  $\lambda'_0=\lambda^{(1)}_0=1$ and $ \|\bom{\lambda}'\rt{L_1}\|)=1$ in the above formula. The second defect is inserted at level $L_1+1$, i.e.\ we modify the sequence into a the sequence
$\bom{\lambda}^{(2)}$
 verifying $\lambda^{(2)}_{y}=\delta_{y,0}+\delta_{y,L_1+1}$,  $\bom{\lambda}^{(2)}\rt{L_1}=\bom{\lambda}^{(1)}$, and $\|\bom{\lambda}^{(2)}\|=2$.
 Again, we can choose a positive integer $L_2>L_1$ such that 
$\sum_{n=1}^{L_2} \BbP[\bom{\lambda}^{(2)}]_0(\bM_n=(0,0))\geq a_2$, and so on. We construct in that way a deterministic sequence 
\[\bom{\lambda}=\lim_{k\rightarrow\infty} \bom{\lambda}^{(k)}
\ \textrm{such that}\  \lambda_y=\delta_{y,0}+\sum_{k=1}^\infty
\delta_{y,L_k+1} \ \textrm{for}\ y\in\BbZ,\] 
verifying  (by construction) $\|\bom{\lambda}\|=\infty$ and $\|\bom{\lambda}\rt{L_k}\|=k$ for all $k$.
For every $k\in\BbN$, during its first $L_k$ steps, the random walk can reach no more than the $k$ first defects and for  $n\leq L_k$, we have therefore $\BbP[\bom{\lambda}]_0(\bM_n=(0,0))=\BbP[\bom{\lambda}\rt{L_k}]_0(\bM_n=(0,0))$. Hence
\begin{eqnarray*}
\sum_{n=0}^\infty \BbP[\bom{\lambda}]_0(\bM_n=(0,0))&\geq&\sum_{n=0}^{L_k} 
\BbP[\bom{\lambda}]_0(\bM_n=(0,0))\\
&=&\sum_{n=0}^{L_k} \BbP[\bom{\lambda}\rt{L_k}]_0(\bM_n=(0,0))\\
&\geq& a_k.
\end{eqnarray*}
Since $k$ is arbitrary, this implies recurrence. 

\cut{
Another interesting question is what happens in more general lattices, like the hexagonal. Since hexagonal lattice can be deformed to be presented as below, we can define random horizontal orientations and ask what will be the type of the walk in this random environment.

\begin{center}
\begin{tikzpicture}
\foreach \y in {-1.5,-0.5, 0, 0.5, 1}  \draw [>->] (-1.7,\y) -- (1.7,\y);
\foreach \y in {-1,  1.5}  \draw [<-<] (-1.7,\y) -- (1.7,\y);
\foreach \y in {-1.5, -0.5, 0.5}  \foreach \x in {-1.5, -0.5, 0.5, 1.5} \draw (\x, \y) -- (\x, \y+0.5);
\foreach \y in {-1.5, -0.5, 0.5}  \foreach \x in {-1.5, -0.5, 0.5} \draw (\x+0.5, \y+0.5) -- (\x+0.5, \y+1);
\end{tikzpicture}
\end{center}
Here the vertical and horizontal components of the random walk no longer factor out completely
as was the case in the square lattice.

Regular 
(undirected) lattices correspond to  Cayley graphs of finitely generated groups $\Gamma$. More precisely, let $\Gamma$ be a finitely generated group, not necessarily Abelian, and $S_\Gamma$ 
a finite symmetric set of generators of $\Gamma$. Then the \textbf{Cayley graph} is the infinite graph
$\textsf{Cayley}(\Gamma, S_\Gamma)=(\BbG^{0},\BbG^{1},s, t)$ with
 $\BbG^{0}=\Gamma$ and $(u,v)\in \BbG^{1}\Leftrightarrow u^{-1} v \in S_\Gamma$. This graph is necessarily undirected and the most prominent examples are the Abelian graphs  $\BbZ^d$ with some integer $d\geq 1$, the homogeneous tree with $d$ free generators $\BbF_d$, etc. 
The construction of the graph can be seen as a recursive nested construction $(\BbG^{0}_n)_{n\in\BbN}$ with $\BbG^{0}_n\subset \BbG^{0}_{n+1}$ for all $n\geq 1$: let $\gamma_0\in $ be some fixed element of $\Gamma$, for instance the neutral element,  identified as a particular vertex of the graph,  and assume $\BbG^{0}_0= \{\gamma_0\}$ be the germ set. Then adjacent vertices are adjoined to get the recursive sequence of sets $\BbG^{0}_{n+1}= \{\gamma s: \gamma\in \BbG^{0}_n, s\in S\}\cup \BbG_n^0$. Now, this construction can be generalised by introducing a \textbf{selection mapping} $F:\Gamma\times S\rightarrow \{0,1\}$; new vertices of the form $\gamma s$, with $s\in S$, adjacent to $\gamma$ can be added, solely\footnote{The Cayley graph case corresponds to the trivial function $F\equiv 1$.} if $F(\gamma,s)=1$. The generated combinatorial object is not any longer a group but merely a groupoid or a semi-groupoid. These constructions occur in a multitude of applications. For instance Penrose lattices obtained from the cut-and-project  method enter into the above groupoid category (diffusive properties of random walks on these lattices are studied in \cite{Telcs2010}, type problems in \cite{Loynes2012c}).
Directed lattices considered in this paper fall into the semi-groupoid one, random graphs are also of the groupoid or semi-groupoid class. The algebraic object  
supporting these (semi)-groupoids are $C^*$-algebras. Therefore, there are interesting counterparts, not yet fully exploited, of the graphs we consider here and various natural objects like Penrose lattices,  Cuntz-Krieger algebras, wavelet cascades, quantum channels, etc. Several of those extensions towards semi-groupoids, quantum channels, and $C^*$-algebras are currently under investigation.
}

\noindent{\textbf{Acknowledgements:}} The authors should like to thank the referee for the very  careful reading and the numerous remarks that helped them improving the presentation of the paper and correcting an erroneous statement in the proof of recurrence.

\bibliographystyle{plain}
\scriptsize{
\bibliography{petritis,complete-bibliography}

\begin{thebibliography}{10}

\bibitem{BhattacharyaWaymire2007}
Rabi Bhattacharya and Edward~C. Waymire.
\newblock {\em A basic course in probability theory}.
\newblock Universitext. Springer, New York, 2007.

\bibitem{Biggs1974}
Norman Biggs.
\newblock {\em Algebraic graph theory}.
\newblock Cambridge University Press, London, 1974.
\newblock Cambridge Tracts in Mathematics, No. 67.

\bibitem{CamPet-rwrol}
Massimo Campanino and Dimitri Petritis.
\newblock Random walks on randomly oriented lattices.
\newblock {\em Markov Process. Related Fields}, 9(3):391--412, 2003.

\bibitem{CamPet-wien}
Massimo Campanino and Dimitri Petritis.
\newblock On the physical relevance of random walks: an example of random walks
  on a randomly oriented lattice.
\newblock In {\em Random walks and geometry}, pages 393--411. Walter de Gruyter
  GmbH \& Co. KG, Berlin, 2004.

\bibitem{Chung1997}
Fan R.~K. Chung.
\newblock {\em Spectral graph theory}, volume~92 of {\em CBMS Regional
  Conference Series in Mathematics}.
\newblock Published for the Conference Board of the Mathematical Sciences,
  Washington, DC, 1997.

\bibitem{CvetkovicDoobSachs1995}
Drago{\v{s}}~M. Cvetkovi{\'c}, Michael Doob, and Horst Sachs.
\newblock {\em Spectra of graphs}.
\newblock Johann Ambrosius Barth, Heidelberg, third edition, 1995.
\newblock Theory and applications.

\bibitem{Loynes2012a}
Basile de~Loynes.
\newblock Marche al\'eatoire sur un di-graphe et fronti\`ere de {M}artin.
\newblock {\em C. R. Math. Acad. Sci. Paris}, 350(1-2):87--90, 2012a.

\bibitem{DevulderP`ene2013}
Alexis Devulder and Fran{\c{c}}oise P{\`e}ne.
\newblock Random walk in random environment in a two-dimensional stratified
  medium with orientations.
\newblock {\em Electron. J. Probab.}, 18:no. 18, 23, 2013.

\bibitem{DoyleSnell1984}
Peter~G. Doyle and J.~Laurie Snell.
\newblock {\em Random walks and electric networks}, volume~22 of {\em Carus
  Mathematical Monographs}.
\newblock Mathematical Association of America, Washington, DC, 1984.

\bibitem{Flanders1971b}
Harley Flanders.
\newblock Infinite networks. {I}: {R}esistive networks.
\newblock {\em IEEE Trans. Circuit Theory}, CT-18:326--331, 1971b.

\bibitem{Flanders1972}
Harley Flanders.
\newblock Infinite networks. {II}. {R}esistance in an infinite grid.
\newblock {\em J. Math. Anal. Appl.}, 40:30--35, 1972.

\bibitem{Guillotin-PlantardLe2007}
N.~Guillotin-Plantard and A.~Le~Ny.
\newblock Transient random walks on 2{D}-oriented lattices.
\newblock {\em Teor. Veroyatn. Primen.}, 52(4):815--826, 2007.

\bibitem{Guillotin-PlantardLe2008}
Nadine Guillotin-Plantard and Arnaud Le~Ny.
\newblock A functional limit theorem for a 2{D}-random walk with dependent
  marginals.
\newblock {\em Electron. Commun. Probab.}, 13:337--351, 2008.

\bibitem{JorgensenPearse2008}
Palle E.~T. Jorgensen and Erin P.~J Pearse.
\newblock Operator theory and analysis of infinite networks, 2008.
\newblock \href{http://arxiv.org/abs/0806.3881}{arXiv:0806.3881}.

\bibitem{MatheronMarsily1980}
G.~Matheron and G.~de~Marsily.
\newblock Is transport in porous media always diffusive? a counter-example.
\newblock {\em Water Resour. Res.}, 16:901--917, 1980.

\bibitem{Pene2009}
Fran{\c{c}}oise P{\`e}ne.
\newblock Transient random walk in {$\Bbb Z^2$} with stationary orientations.
\newblock {\em ESAIM Probab. Stat.}, 13:417--436, 2009.

\bibitem{Pete2008}
G{\'a}bor Pete.
\newblock Corner percolation on {$\Bbb Z^2$} and the square root of 17.
\newblock {\em Ann. Probab.}, 36(5):1711--1747, 2008.

\bibitem{Port1994}
Sidney~C. Port.
\newblock {\em Theoretical probability for applications}.
\newblock Wiley Series in Probability and Mathematical Statistics: Probability
  and Mathematical Statistics. John Wiley \& Sons Inc., New York, 1994.
\newblock A Wiley-Interscience Publication.

\bibitem{Redner1997}
Sidney Redner.
\newblock Survival probability in a random velocity field.
\newblock {\em Phys. Rev. E}, 56:4967--4972, 1997.

\bibitem{Soardi1994}
Paolo~M. Soardi.
\newblock {\em Potential theory on infinite networks}, volume 1590 of {\em
  Lecture Notes in Mathematics}.
\newblock Springer-Verlag, Berlin, 1994.

\bibitem{Weyl1923}
Hermann Weyl.
\newblock Repartici\'{o}n de corriente en una red conductora.
\newblock {\em Revista Matem\'{a}tica Hispano-Americana}, 5:153--164, 1923.

\bibitem{Woess2000}
Wolfgang Woess.
\newblock {\em Random walks on infinite graphs and groups}, volume 138 of {\em
  Cambridge Tracts in Mathematics}.
\newblock Cambridge University Press, Cambridge, 2000.

\end{thebibliography}
 }            

\end{document}